\documentclass[11pt]{article}
\usepackage{amsmath, amsthm, amssymb, amsfonts, latexsym,amscd}
\usepackage{authblk}
\usepackage{tikz}
\usetikzlibrary{graphs,graphs.standard} 
\usepackage{mathtools}
\usetikzlibrary{shadings,patterns}
\newtheorem{theorem}{Theorem}[section]
\newtheorem{corollary}[theorem]{Corollary}
\newtheorem{lemma}[theorem]{Lemma}
\newtheorem{proposition}[theorem]{Proposition}

\newtheorem{remark}[theorem]{Remark}
\newtheorem{example}[theorem]{Example}

\newtheorem{conjecture}[theorem]{Conjecture}
\begin{document}

\title{\bf On the sum of eccentric subgraph numbers of the vertices of a connected graph}

\author[1]{Akash De \thanks{Corresponding Author: akash.de.2023@niser.ac.in} }
\author[2]{Kamal Lochan Patra \thanks{klpatra@niser.ac.in} }
\affil[1,2] {School of Mathematical Sciences,
National Institute of Science Education and Research Bhubaneswar,
P.O.- Jatni, District- Khurda, Odisha - 752050, India 
}
\affil[1,2] { Homi Bhabha National Institute(HBNI), 
Training School Complex, Anushakti Nagar, Mumbai - 400094, India 
}
\date{}
\maketitle
\begin{abstract}
Let $G$ be a connected graph. The eccentric subgraph number $\epsilon(v)$ of a vertex $v$ of $G$ is defined as $\epsilon(v)=\min\{f_G(v,u): u\in V(G)\}$, where $f_G(v,u)$ is the number of connected subgraphs of $G$ containing both $v$ and $u$. The  eccentric subgraph index $ES(G)$ of $G$ is  defined as $ES(G)=\displaystyle\sum_{v\in V(G)}\epsilon(v).$ In this paper, we obtain the unicyclic graph that maximizes the eccentric subgraph index among all unicyclic graphs on $n$ vertices with a fixed girth. The graph that maximizes the eccentric subgraph index among all unicyclic graphs on $n$ vertices is also obtained. We then characterize the trees that maximize the eccentric subgraph index over trees on $n$ vertices with fixed diameter. Among all connected graphs on $n$ vertices with a fixed number of pendant vertices, the graph which maximizes the eccentric subgraph index is also characterized.


\end{abstract}

\section{Introduction}
Throughout this paper, graphs are undirected, finite, simple and connected. Let $G$ be a graph with vertex set $V(G)$ and edge set $E(G)$. A vertex $v$ of $G$ is called a  pendant vertex if degree of $v$ is one. The distance $d(u, v)$ between two vertices $u$ and $v$ of $G$ is defined as the number of edges in a shortest path joining them. The  diameter diam$(G)$ of $G$ is defined as diam$(G)=\max\{d(u,v)|u,v\in V(G)\}.$ A graph $G$ on $n$ vertices is said to be unicyclic if it has $n$ edges. The girth of $G$ is the length of a shortest cycle in $G$. For $v_1,v_2,\ldots,v_k\in V(G)$, $f_G(v_1,v_2,\ldots,v_k)$ denotes the number of connected subgraphs of $G$ containing $v_1,v_2,\ldots,v_k$. By $f_G(v_1,\ldots,v_{k-1},\overline{v_k})$, we mean the number of connected subgraph of $G$ containing $v_1,\ldots,v_{k-1}$ but not $v_k$.

For a vertex $v$ of  $G$, the {\it eccentric subgraph number} $\epsilon_G(v)$(or simply $\epsilon(v)$) is defined as $\epsilon_G(v)=\min\{f_G(v,u): u\in V(G)\}$.  A {\it core vertex} of $G$ is a vertex with the maximum eccentric subgraph number, and the set of all core vertices is called the {\it core center} of $G$. We denote the  core center of $G$ by $\mathfrak{C}(G)$. This is first defined in \cite{Pandey}, which is a generalization of eccentric subtree number for trees(see \cite{Zhang}) to connected graphs. The core center $\mathfrak{C}(G)$ shares certain common properties with other central parts of $G$ (see \cite{Pandey}). The {\it eccentric subgraph index} of $G$ is denoted by $ES(G)$ and defined as $ES(G)=\displaystyle\sum_{v\in V(G)}\epsilon(v).$ We consider $ES(G)$ as the index associated with the central part $\mathfrak{C}(G)$ of $G$. We are interested in obtaining the extremal graphs associated with the eccentric subgraph index. Following is the first and only known result in this direction.
\begin{theorem}[\cite{Zhang}, Theorem 4.3]\label{tree}
Among all trees on n vertices, the maximum eccentric subgraph index is attained by the star and the minimum eccentric subgraph index is attained by the path.
\end{theorem}
For the star  $K_{1,n-1}$,  $n\geq 3$ we have, $\epsilon(v)=2^{n-2}$ where $v\in V(K_{1,n-1})$ is the vertex of degree $n-1$ and $\epsilon(u)=2^{n-3}$ for any other vertex $u$ of $K_{1,n-1}$. Thus, $$ES(K_{1,n-1})=2^{n-2}+(n-1)2^{n-3}=(n+1)2^{n-3}.$$ For the path $P_n:12\cdots n$, $n\geq 2$ we have  $\epsilon(k)=f_{P_n}(k,n)=k=\epsilon(n-k+1) \;\; \mbox{for} \;\; 1\leq k\leq \left\lceil\frac{n}{2}\right\rceil$. Thus,   
 
 $$ES(P_n)=\begin{cases}
(\frac{n}{2})(\frac{n}{2}+1) &\mbox{if $n$ is even}; \\
(\frac{n+1}{2})^2 & \mbox{if $n$ is odd.}
\end{cases}$$

\subsection{Motivation and main results}

There are several indices defined on a connected graph $G$ and some of these indices are related to a central part of $G$. The distance
$D_G(v)$ of a vertex $v$ in $G$ is defined as $D_G(v) =\displaystyle\sum_{u\in V(G)} d(v,u)$. A vertex of minimum distance is called a median vertex of
$G$ and the set of all median vertices is called the median of $G$. The Wiener index of $G$, denoted by $W(G)$, is defined as the sum of distances
between all unordered pairs of its vertices, i.e.,
$$W(G)=\displaystyle\sum_{\{u,v\}\subseteq V(G)} d(v,u)=\frac{1}{2}\displaystyle\sum_{v\in V(G)}D_G(v).$$
The Wiener index of $G$ is associated with the central part median of $G$. A lot has been studied on Wiener index focusing on the extremal graphs associated with it.
To know more about Wiener index and its associated extremal graphs, we refer to some recent works (see \cite{Pandey2,Xu}) and the references therein. 

The eccentricity of a vertex $v$ of $G$, denoted by $e(v)$, is defined as $e(v)=\max\{d(v,u)|u\in V(G)\}$. A vertex of minimum eccentricity in $G$ is called a central vertex and the set of all central vertices of $G$ is called the center of $G$. The total eccentricity index of $G$, denoted by $TE(G)$, is defined as the sum of eccentricities of all its vertices, i.e.,
$$TE(G)=\displaystyle\sum_{v\in V(G)} e(v).$$
The total eccentricity index is also studied by the name of eccentricity sum and total eccentricity (see \cite{Fat,Smith}). The total eccentricity index is associated with the central part center of $G$. For a more comprehensive understanding of the total eccentricity index and its associated extremal graphs, we refer the reader to a recent study (see \cite{Pandey3}) and the references therein.

The subgraph number of a vertex $v$ of $G$, denoted by $f_G(v)$, is defined as the number of connected subgraphs of $G$ containing $v$. The set of vertices of $G$ having maximum subgraph number is called the subgraph core of $G$. The number of connected subgraphs of $G$ is called the subgraph index of $G$ (see \cite{Pandey4}). The subgraph core of $G$ is considered to be a central part of $G$ (see \cite{Pandey}) and subgraph index is the index associated to it. For more on subgraph index of $G$ and its associated extremal graphs, we refer to \cite{Pandey4} and the references therein.

Motivated by these three above indices of $G$ and their association with different central parts of $G$, we call the sum $\displaystyle\sum_{v\in V(G)}\epsilon(v)$ as eccentric subgraph index of $G$. It is interesting to know about some extremal graphs associated with the eccentric subgraph index. In this paper, we obtain the following results:
\begin{enumerate}
\item[(i)] We obtain the unique graph that maximizes the eccentric subgraph index among all unicyclic graphs on $n$ vertices with a fixed girth.
\item[(ii)] We obtain the unique graph that maximizes the eccentric subgraph index among all unicyclic graphs of order $n$.
\item[(iii)] We characterize  the trees on $n$ vertices with a fixed diameter that maximize the eccentric subgraph index.
\item[(iv)] We obtain the unique graph maximizing the eccentric subgraph index among all connected graphs on $n$ vertices with a fixed number of pendant vertices. 
\end{enumerate}

\section{Preliminaries}\label{Prim}

We start this section with the following lemma.
\begin{lemma}\label{edge}
Let $G$ be a graph and let $u, v \in V(G)$ be two non-adjacent vertices. Let $\tilde{G}$ be the graph obtained from $G$ by joining $u$ and $v$ with an edge. Then $ES(G)<ES(\tilde{G})$.
\end{lemma}
\begin{proof}
Let $e=\{u,v\}$ be the edge of $\tilde{G}$ joining $u$ and $v$. For $x\in V(\tilde{G})$, $\epsilon_{\tilde{G}}(x)=f_{\tilde{G}}(x,w)$ for some $w\in V(\tilde{G})$. 
We have $f_{\tilde{G}}(x,w)=f_{\tilde{G}}^{e+}(x,w)+ f_{\tilde{G}}^{e-}(x,w)$, where $f_{\tilde{G}}^{e+}(x,w)$ is the number of connected subgraph of $\tilde{G}$ 
containing $x$, $w$ and the edge $e$ and $f_{\tilde{G}}^{e-}(x,w)$ is the number of connected subgraph of $\tilde{G}$ containing $x$, $w$ but not the edge $e$. Since $f_{\tilde{G}}^{e-}(x,w)=f_{G}(x,w)$ and $f_{\tilde{G}}^{e+}(x,w)>0$, we have $\epsilon_{\tilde{G}}(x)>f_{G}(x,w)\geq \epsilon_{G}(x)$. Thus $$ES(\tilde{G})=\displaystyle\sum_{x\in V(\tilde{G})}\epsilon_{\tilde{G}}(x)>\displaystyle\sum_{x\in V(G)}\epsilon_{G}(x)=ES(G).$$
\end{proof}
By Lemma \ref{edge} and Theorem \ref{tree}, we have the following results.
\begin{theorem}\label{thm1}
Among all graphs on $n$ vertices, the eccentric subgraph index is maximized by the complete graph $K_n$  and minimized by the path $P_n$.
\end{theorem}

\begin{figure}[h!]
    \centering
  \begin{tikzpicture}[scale=1.2]
  \draw(-11,0) circle(0.7cm);
  \draw(-8.6,0) circle(0.7cm);
   \draw(-6,0) circle(0.7cm);
  \draw(-4.59,0) circle(0.7cm);
 \node[draw,circle,fill=black,scale=0.2](1)at (-10.3,0){};
  \node[draw,circle,fill=black,scale=0.2](2)at (-9.3,0){};
  \node[draw,circle,fill=black,scale=0.2](3)at (-5.3,0){};
   \node[draw,circle,fill=black,scale=0.2](4)at (-5.3,1){};
   \node (a) at (-9.85,0.12){};
    \node (b) at (-9.75,0.12){};
   \node (5) at (-11,-0.3){};
    \node (6) at (-8.5,-0.25){};
    \node (7) at (-6.2,-0.2){};
    \node (8) at (-4.57,-0.2){};
    \node (9) at (-9.8,-1){};
    \node (10) at (-5.2,-1.2){};
    
 \draw(1)--(2);
 \draw(3)--(4);
 \node[anchor=east] at (a) {$v_1$};
 \node[anchor=west] at (b) {$v_2$};
\node[anchor=east] at (3) {$v_1$};
\node[anchor=south] at (4) {$v_2$}; 
\node[anchor=south] at (5) {$G_1$}; 
\node[anchor=south] at (6) {$G_2$}; 
\node[anchor=south] at (7) {$G_1$}; 
\node[anchor=south] at (8) {$G_2$};
\node[anchor=south] at (9) {$G$}; 
\node[anchor=south] at (10) {$G'$};
\end{tikzpicture} 
\caption{Graphs $G$ and $G'$}\label{fig1}
\end{figure}
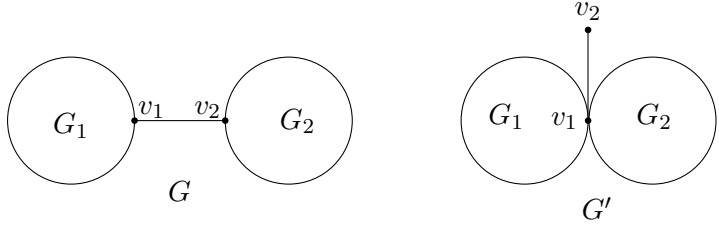

\begin{lemma}\label{lem1}
Let $G$ be a graph with a cut edge $e=\{v_1,v_2\}$. Let $G_1$ and $G_2$ be the  connected components of $G-e$ containing $v_1$ and $v_2$, respectively, with 
$|V(G_i)|\geq 2$ for $i=1,2$. Let $G'$ be the graph obtain from $G$ by identifying $v_1$ and $v_2$ and adding a pendant  vertex $v_2$  at the identified vertex  $v_{1}$(see Figure \ref{fig1}). Then $\epsilon_G(x)<\epsilon_{G'}(x)$ for all $x\in V(G)- v_2$.
\end{lemma}
\begin{proof}
Let $G_1$ and $G'_2$ be the subgraphs  of $G'$ such that $G_1\cup G'_2=G'$ and $V(G_1)\cap V(G'_2)=\{v_1\}$. 

For $x\in V(G_1)$, let $\epsilon_{G'}(x)=f_{G'}(x,r)$. Clearly $r\neq v_1$.\\ 
Suppose $r\in V(G'_2)$. Then $f_{G'}(x,r)=f_{G_1}(x,v_1)f_{G'_2}(v_1,r)$.  We have two cases here.\\
{\bf Case-I:} $r\neq v_2$\\
Since $ 2f_{G_2}(v_2,r)= f_{G'_2}(v_1,r)$, we have
\begin{align*}
   &\epsilon_{G}(x)\leq f_G(x,r)=f_{G_1}(x,v_1)f_{G_2}(v_2,r)< f_{G_1}(x,v_1)f_{G'_2}(v_1,r)= f_{G'}(x,r)=\epsilon_{G'}(x). 
\end{align*} 
\noindent {\bf Case-II:} $r= v_2$\\
 In this case, we have 
\begin{align*}
   &\epsilon_{G}(x)< f_G(x,v_2)=f_{G_1}(x,v_1)f_{G_2}(v_2)= f_{G_1}(x,v_1)f_{G'_2}(v_1,v_2)= f_{G'}(x,v_2)=\epsilon_{G'}(x). 
\end{align*}
Suppose $r\in V(G_1)$. Then $f_{G'}(x,r)=f_{G_1}(x,r)+f_{G_1}(x,r,v_1)(f_{G'_2}(v_1) -1)$ and $f_{G}(x,r)=f_{G_1}(x,r)+f_{G_1}(x,r,v_1)f_{G_2}(v_2)$. Since $f_{G'_2}(v_1)=2f_{G_2}(v_2)$, we have
\begin{align*}
 f_{G'}(x,r)-f_{G}(x,r)=f_{G_1}(x,r,v_1)(f_{G'_2}(v_1)-1-f_{G_2}(v_2))>0 \;\;(\mbox{as $|V(G_2)|>1$}) .
\end{align*}
So, $\epsilon_{G}(x)\leq f_{G}(x,r)<f_{G'}(x,r)=\epsilon_{G'}(x)$. Hence $\epsilon_{G}(x)<\epsilon_{G'}(x)$ for $x\in V(G_1)$.\\

Now let $x\in V(G'_2)\setminus \{v_1,v_2\}$ and suppose $\epsilon_{G'}(x)=f_{G'}(x,r)$. Clearly $r\neq v_1$.\\
Suppose $r\in V(G_1)$. Then $f_{G'}(x,r)=f_{G'_2}(x,v_1)f_{G_1}(v_1,r)$. Since $f_{G'_2}(x,v_1)=2f_{G_2}(x,v_2)$, 
 \begin{align*}
  \epsilon_{G}(x)\leq f_G(x,r)=f_{G_2}(x,v_2)f_{G_1}(v_1,r)< f_{G'_2}(x,v_1)f_{G_1}(v_1,r)=f_{G'}(x,r)=\epsilon_{G'}(x).
 \end{align*}
Suppose $r\in V(G'_2)$. Then $f_{G'}(x,r)=f_{G'_2}(x,r,\bar{v_1})+f_{G'_2}(x,r,{v_1})+f_{G'_2}(x,r,v_1)(f_{G_1}(v_1)-1)$ and
    $f_G(x,r)=f_{G_2}(x,r,\bar{v_2})+f_{G_2}(x,r,v_2)+f_{G_2}(x,r,v_2)f_{G_1}(v_1)$.
Since $f_{G'_2}(x,r,\bar{v_1})=f_{G_2}(x,r,\bar{v_2})$ and $f_{G'_2}(x,r,{v_1})=2f_{G_2}(x,r,v_2)$, we have
\begin{align*}
   f_{G'}(x,r)-f_G(x,r)&=f_{G'_2}(x,r,v_1)f_{G_1}(v_1)-f_{G_2}(x,r,v_2)-f_{G_2}(x,r,v_2)f_{G_1}(v_1)\\
                       &=2f_{G_2}(x,r,v_2)f_{G_1}(v_1)-f_{G_2}(x,r,v_2)-f_{G_2}(x,r,v_2)f_{G_1}(v_1)\\
                       &=f_{G_2}(x,r,v_2)(f_{G_1}(v_1)-1)\\
                       &>0 \;\;(\mbox{as $|V(G_1)|>1$}) .
\end{align*}
So, $\epsilon_{G}(x)\leq f_G(x,r)< f_{G'}(x,r)=\epsilon_{G'}(x)$. Hence $\epsilon_{G}(x)<\epsilon_{G'}(x)$ for $x\in V(G'_2)\setminus \{v_1,v_2\}$ and this proves the result.
\end{proof}
In the next example we have shown that the result in Lemma \ref{lem1} may not be true for the vertex $v_2$, i.e. $\epsilon_{G}(v_2)$ may not be always less than $\epsilon_{G'}(v_2)$.

\begin{figure}[h!]
    \centering
  \begin{tikzpicture}[scale=0.7]
   \node[draw,circle,fill=black,scale=0.2,label=above:\(3\)](1)at (0,0){};
   \node[draw,circle,fill=black,scale=0.2,label=above:\(4\)](2)at (2,1){};
   \node[draw,circle,fill=black,scale=0.2,label=below:\(5\)](3)at (2,-1){};
   \node[draw,circle,fill=black,scale=0.2,label=above:\(2\)](4)at (-1.25,0){};
   \node[draw,circle,fill=black,scale=0.2,label=above:\(1\)](5)at (-2.5,0){};
    \node[draw,circle,fill=black,scale=0.2,label=above:\(6\)](6)at (2.9,1.5){};    \node[draw,circle,fill=black,scale=0.2,label=above:\(7\)](7)at (4,2.1){};
  \draw(1)--(2)--(3)--(1);
  \draw(1)--(4)--(5);
  \draw(2)--(6)--(7);

  \node[draw,circle,fill=black,scale=0.2,label=below:\(2\)](8)at (6,0){};
   \node[draw,circle,fill=black,scale=0.2,label=above:\(4\)](9)at (8,1){};
   \node[draw,circle,fill=black,scale=0.2,label=below:\(5\)](10)at (8,-1){};
   \node[draw,circle,fill=black,scale=0.2,label=above:\(1\)](11)at (4.5,0){};
   \node[draw,circle,fill=black,scale=0.2,label=above:\(3\)](12)at (6,1.2){};
    \node[draw,circle,fill=black,scale=0.2,label=above:\(6\)](13)at (8.9,1.5){};    \node[draw,circle,fill=black,scale=0.2,label=above:\(7\)](14)at (10,2.1){};
  \draw(8)--(9)--(10)--(8);
  \draw(8)--(11);
  \draw(8)--(12);
  \draw(9)--(13)--(14);
  \node (a) at (0.5,-1.5){};
   \node (b) at (6.2,-1.5){};
   \node[anchor=south] at (a) {$G$}; 
\node[anchor=south] at (b) {$G'$}; 
\end{tikzpicture} 
\caption{Eccentric subgraph number of $3$}\label{fig2}
\end{figure}
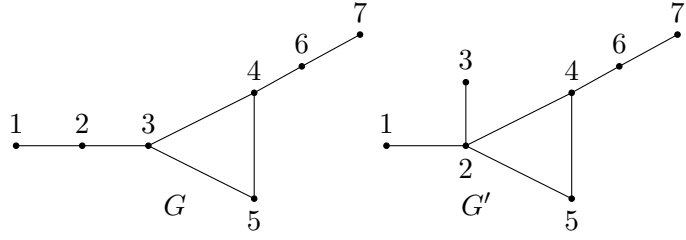
\begin{example}
Consider the graph $G$ in Figure $\ref{fig2}$, where $e=\{2,3\}$  is a cut edge. Then $\epsilon_{G}(3)=f_{G}(3,7)=15$. We obtain $G'$ by identifying  vertices $2$ and $3$ and attaching pendant vertex $3$ at identified vertex $2$. Then $\epsilon_{G'}(3)=f_{G'}(3,7)=10$. So, $\epsilon_{G'}(3)<\epsilon_{G}(3)$.
\end{example}
\begin{lemma}\label{lem2}
Let $r$ be a pendant vertex of a tree $T$ and let $v_1\in V(T)$. Suppose $P:v_1v_2\cdots v_dr$ is the $v_1$-$r$ path in $T$. Then $f_T(v_1)-2f_T(v_1,r)\geq 0$ and the equality happens if and only if $d=1$. 
\end{lemma}
\begin{proof}
Let $X_1$ be the component of $T-v_1v_2$ containing $v_1$, $X_2$ be the component of  $T-v_1v_2-v_2v_3$ containing $v_2$, $\ldots$ , $X_d$ be the component of  $T-v_{d-1}v_d-v_dr$ containing $v_d$ (see Figure \ref{fig3}). Then
\begin{align*}
    f_T(v_1,r)=\prod_{i=1}^{d} f_{X_i}(v_i),\;\; f_T(v_1,v_d,\overline{r})=\prod_{i=1}^{d} f_{X_i}(v_i)
\end{align*} 
and for $1\leq j \leq d-1$
\begin{align*}
   f_T(v_1,v_j,\overline{v_{j+1}})= \prod_{i=1}^{j} f_{X_i}(v_i).
\end{align*}
So,
\begin{align*}
 f_T(v_1)&=f_{X_1}(v_1)+f_{X_1}(v_1)f_{X_2}(v_2)+\cdots+2f_{X_1}(v_1)f_{X_2}(v_2)\cdots f_{X_d}(v_d)\\
         &=f_{X_1}(v_1)+f_{X_1}(v_1)f_{X_2}(v_2)+\cdots+2f_T(v_1,r).
\end{align*} 
  Thus, $f_T(v_1)-2f_T(v_1,r)\geq 0$ and equality happens if and only if $f_{X_1}(v_1)+f_{X_1}(v_1)f_{X_2}(v_2)+\cdots+f_{X_1}(v_1)f_{X_2}(v_2)\cdots f_{X_{d-1}}(v_{d-1})=0$ and hence equality happens if and only if $d=1.$
\end{proof}
\begin{figure}[h!]
    \centering
 \begin{tikzpicture}[scale=1.0, every node/.style={scale=0.8}]
    \node[draw,circle,fill=black, inner sep=1.5pt, label=above:\(v_1\)](m)at (1.5,0){};
    \node[draw,circle,fill=black, inner sep=1.5pt, label=above:\(v_2\)](m1)at (4.5,0){};
    \node[draw,circle,fill=black, inner sep=1.5pt, label=above:\(v_d\)](m2)at (7.5,0){};
    \node[draw,circle,fill=black, inner sep=1.5pt, label=above:\(r\)](r)at (9,0){};
    \draw(m)--(m1);
    \draw[dashed](m1)--(m2);
    \draw(m2)--(r);
    \draw(1.5,0)--(1,-1)--(2,-1)--cycle;
    \node at (1.5,-0.8){\(X_1\)};
    \draw(4.5,0)--(4,-1)--(5,-1)--cycle;
    \node at (4.5,-0.8){\(X_2\)};
    \draw(7.5,0)--(7,-1)--(8,-1)--cycle;
    \node at (7.5,-0.8){\(X_d\)};
    \end{tikzpicture}
    \caption{The tree T of Lemma \ref{lem2}}\label{fig3}
    
\end{figure}
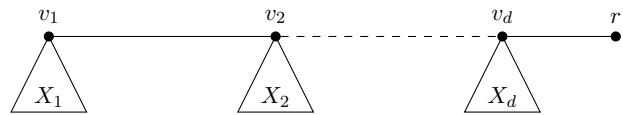

\begin{lemma}\label{lem3}
Let $G$ and $G'$ be the graphs mentioned in Lemma \ref{lem1} (Figure \ref{fig1}). Suppose  $\epsilon_{G'}(v_2)=f_{G'}(v_2,r)$. Then the following hold:
\begin{enumerate}
    \item[(i)] $\epsilon_{G'}(v_1)=min\{f_{G'}(v_1,v_2),f_{G'}(v_1,r)\}$;
    \item[(ii)] If $G_1$ is a tree then $\epsilon_{G'}(v_1)=f_{G'}(v_1,r)$. 
\end{enumerate}
\end{lemma}
\begin{proof} 
\begin{enumerate}
\item[(i)] Let $G_1$ and $G'_2$ be the subgraphs  of $G'$ such that $G_1\cup G'_2=G'$ and $V(G_1)\cap V(G'_2)=\{v_1\}$. Clearly $r\neq v_2$. Suppose $r\in V(G'_2)\setminus v_2$. 
    
If $y\in V(G_1)$ then $\epsilon_{G'}(v_2)=f_{G'}(v_2,r)$ implies $f_{G'_2}(v_2,r)f_{G_1}(v_1)\leq f_{G'_2}(v_2,v_1)f_{G_1}(v_1,y)$. Since $\frac{f_{G'_2}(v_1,r)}{2}=f_{G'_2}(v_2,r)$ and $\frac{f_{G'_2}(v_1)}{2}=f_{G'_2}(v_2,v_1)$,  $\frac{f_{G'_2}(v_1,r)}{2}f_{G_1}(v_1)\leq \frac{f_{G'_2}(v_1)}{2}f_{G_1}(v_1,y)$. So,
\begin{align}\label{a1}
  f_{G'}(v_1,r) =f_{G'_2}(v_1,r)f_{G_1}(v_1)\leq f_{G'_2}(v_1)f_{G_1}(v_1,y)=f_{G'}(v_1,y).
\end{align}
If $y\in V(G'_2)\setminus v_2$ then $f_{G'_2}(v_2, r)\leq f_{G'_2}(v_2,y)$. Since $\frac{f_{G'_2}(v_1,r)}{2}=f_{G'_2}(v_2,r)$ and $\frac{f_{G'_2}(v_1,y)}{2}=f_{G'_2}(v_2,y)$, we have  $\frac{f_{G'_2}(v_1,r)}{2}f_{G_1}(v_1)\leq \frac{f_{G'_2}(v_1,y)}{2}f_{G_1}(v_1)$. Thus
\begin{align}\label{a2}
   f_{G'}(v_1,r)=f_{G'_2}(v_1,r)f_{G_1}(v_1)\leq f_{G'_2}(v_1,y)f_{G_1}(v_1)=f_{G'}(v_1,y). 
\end{align}
By $(\ref{a1})$ and $(\ref{a2})$ we get $\epsilon_{G'}(v_1)=min\{f_{G'}(v_1,v_2),f_{G'}(v_1,r)\}$ if $r\in V(G'_2)\setminus v_2$. Similarly, we can prove that $\epsilon_{G'}(v_1)=min\{f_{G'}(v_1,v_2),f_{G'}(v_1,r)\}$ if $r\in V(G_1)$. This proves the result.

\item [(ii)]  Let $x$  be a pendant vertex of $G_1$. Then $f_{G'}(v_1,x)=f_{G_1}(v_1,x)f_{G'_2}(v_1)$ and $f_{G'}(v_1,v_2)=f_{G_1}(v_1)f_{G'_2}(v_1,v_2)$. Since $f_{G'_2}(v_1)=2f_{G_2}(v_2)$ and $f_{G'_2}(v_1,v_2)=f_{G_2}(v_2)$, we have
\begin{align*}
  f_{G'}(v_1,v_2)-f_{G'}(v_1,x)&=f_{G_1}(v_1)f_{G_2}(v_2)-2f_{G_1}(v_1,x)f_{G_2}(v_2)\\&=(f_{G_1}(v_1)-2f_{G_1}(v_1,x))f_{G_2}(v_2)\\& \geq 0 \;\; (\mbox{by Lemma $\ref{lem2}).$}
\end{align*}
Thus $f_{G'}(v_1,r)\leq f_{G'}(v_1,x)\leq f_{G'}(v_1,v_2)$ where the first inequality follows from $(\ref{a1})$ and hence from part $(i)$ the result follows.
\end{enumerate} 
\end{proof}

\begin{theorem}\label{thm2}
Let $G$ and $G'$ be the graphs mentioned in Lemma \ref{lem1} (Figure \ref{fig1}). If $G_1$ is a tree then $ES(G)<ES(G')$.
\end{theorem}
\begin{proof}
Let $G_1$ and $G'_2$ be the subgraphs  of $G'$ such that $G_1\cup G'_2=G'$ and $V(G_1)\cap V(G'_2)=\{v_1\}$. Suppose $\epsilon_{G'}(v_2)=f_{G'}(v_2,s)$. We have two cases here.\\
{\bf Case-I:} $s\in V(G_1)$\\
In this case, we have 
\begin{align*} 
 \epsilon_{G}(v_2)\leq f_{G}(v_2,s)&=f_{G_2}(v_2)f_{G_1}(v_1,s)\\
                                   &=f_{G'_2}(v_2,v_1)f_{G_1}(v_1,s)\\
                                   &=f_{G'}(v_2,s)\\
                                   &=\epsilon_{G'}(v_2).
\end{align*}
 Now the result follows from Lemma \ref{lem1}.
 
\noindent {\bf Case-II:} $s\in V(G'_2)\setminus v_2$ \\
In this case,
$f_{G'}(v_2,s)=f_{G'_2}(v_2,s)f_{G_1}(v_1)=f_{G_2}(v_2,s)f_{G_1}(v_1)$, $f_{G'}(v_1,s)=f_{G'_2}(v_1,s)f_{G_1}(v_1)=2f_{G_2}(v_2,s)f_{G_1}(v_1)$,
  $f_{G}(v_2,s)=f_{G_2}(v_2,s)(1+f_{G_1}(v_1))$    and  $f_{G}(v_1,s)=f_{G_2}(v_2,s)f_{G_1}(v_1).$ Thus
\begin{align*}
   &f_{G}(v_1,s)+f_{G}(v_2,s)= f_{G_2}(v_2,s)(1+2f_{G_1}(v_1)),\\
   &f_{G'}(v_1,s)+f_{G'}(v_2,s)=3f_{G_2}(v_2,s)f_{G_1}(v_1).
\end{align*}
Since $G_1$ is a tree,  by Lemma $\ref{lem3}$,  we have $\epsilon_{G'}(v_1)=f_{G'}(v_1,s)$. Hence,
\begin{align*}
    \epsilon_{G}(v_{1})+\epsilon_{G}(v_2)\leq f_{G}(v_2,s)+f_{G}(v_1,s) <f_{G'}(v_1,s)+f_{G'}(v_2,s)= \epsilon_{G'}(v_1)+\epsilon_{G'}(v_2),
\end{align*}
and the result follows from Lemma $\ref{lem1}$.   
\end{proof}
\begin{corollary}
Among all trees on $n$ vertices, the star $K_{1,n-1}$ has the maximum eccentric subgraph index.
\end{corollary}
We believe that the result in Theorem \ref{thm2} is true for any connected graphs $G_1$ and $G_2$ and hence we are proposing the following conjecture.
\begin{conjecture}
Let $G$ be a graph with a cut edge $e=\{v_1,v_2\}$. Let $G_1$ and $G_2$ be the  connected components of $G-e$ containing $v_1$ and $v_2$, respectively, with 
$|V(G_i)|\geq 2$ for $i=1,2$. Let $G'$ be the graph obtained from $G$ by identifying $v_1$ and $v_2$ and adding a pendant  vertex $v_2$  at the identified vertex  $v_{1}$. Then $ES(G)< ES(G').$
\end{conjecture} 
 
\section{Unicyclic graphs}\label{ucg}

Let $n\geq 3$ and $3\leq g\leq n$. We denote the set of all unicyclic graphs on $n$ vertices by $\mathcal{U}_n$ and the set of all unicyclic graphs on $n$  vertices with girth $g$ by $\mathcal{U}_{n,g}$. We obtain the graph $U^p_{n,g}$  by attaching $n-g$ pendant vertices at a vertex of the cycle $C_g$ (see Figure \ref{fig4}). Clearly $U^p_{n,g}\in \mathcal{U}_{n,g}$. We start this section by determining $ES(C_n)$.
\begin{figure}[h!]
    \centering
  \begin{tikzpicture}[scale=1.2]
  \draw(-11,0) circle(0.7cm);
  \node[draw,circle,fill=black,scale=0.2](1)at (-10.32,0.2){};
  \node[draw,circle,fill=black,scale=0.2] (a) at (-9.5,0.8){};
  \node[draw,circle,fill=black,scale=0.2] (b) at (-9.4, 0.6){};
  \node[draw,circle,fill=black,scale=0.2] (c) at (-9.3, 0.4){};
  \node[draw,circle,fill=black,scale=0.2] (d) at (-9.35,-0.1){};
   \node (2) at (-10.7,0.0){};
  \draw(1)--(a);
  \draw(1)--(b);
  \draw(1)--(c);
  \draw(1)--(d);
  \draw[dashed](c) to [<->, out=-65, in=60](d);
  \node [anchor=center] at (-8.8, 0.4) {\rotatebox{90}{$\underbrace{\hspace{1.5cm}}_{n-g}$}};
  \node[anchor=east] at (2) {$C_g$};
 \end{tikzpicture}
  \caption{The  graph $U^p_{n,g}$}\label{fig4}
 \end{figure}
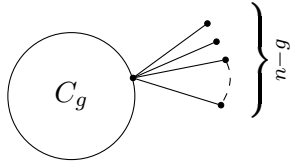

\begin{proposition}\label{prop1}
Let $C_n:v_1v_2 \cdots v_nv_1$ be the cycle on $n\geq 3$ vertices. Then $\epsilon_{C_n}(v_i)=\lfloor \frac{n^2+2n+5}{4}\rfloor$ for $i=1,2,\ldots,n$. Furthermore, $ES(C_n)=n\lfloor \frac{n^2+2n+5}{4}\rfloor$.
\end{proposition}
\begin{proof}
Since all vertices of $C_n$ have the same eccentric subgraph number,  we will only compute $\epsilon(v_1)=\min\{f_{C_n}(v_1,v_k),\ 2\leq k\leq n\}$. Observe that $C_n$ itself a connected subgraph containing $v_1$ and $v_k$ and  all other connected subgraphs  of $C_n$ containing $v_1$ and $v_k$ are the paths containing $v_1$ and $v_k$. There are exactly two paths of $C_n$ containing both $v_1$ and $v_k$ as  pendant vertices, one  is $v_1v_2 \cdots v_k$ and another is $v_kv_{k+1} \cdots v_nv_{1}$. Also, we have $2(n-k)+2(k-2)=2(n-2)$ paths of $C_n$ containing both $v_1$ and $v_k$, and either $v_1$ or $v_k$ as a pendant vertex. The number of paths of $C_n$ containing both $v_1$ and $v_k$ as non-pendant vertices is $\binom{n-k}{2}+\binom{k-2}{2}$. Thus $f_{C_n}(v_1,v_k)=3+2(n-2)+\binom{n-k}{2}+\binom{k-2}{2}=2n-1+\binom{n-k}{2}+\binom{k-2}{2}$. So,
\begin{align*}
  f_{C_n}(v_1,v_k)- f_{C_n}(v_1,v_{k+1})&=\left[\binom{n-k}{2}-\binom{n-k-1}{2}\right]-\left[\binom{k-1}{2}-\binom{k-2}{2}\right] \\
   &=\binom{n-k-1}{1}-\binom{k-2}{1} \\
   &=n-2k+1.
\end{align*}
Hence, $f_{C_n}(v_1,v_k)> f_{C_n}(v_1,v_{k+1})$  if $k<\frac{n+1}{2}$ and $f_{C_n}(v_1,v_k)< f_{C_n}(v_1,v_{k+1})$ if $k>\frac{n+1}{2}$.

Suppose $n$ is even. Then $f_{C_n}(v_1,v_k)> f_{C_n}(v_1,v_{k+1})$  if  $k\leq \frac{n}{2}$ and $f_{C_n}(v_1,v_k)< f_{C_n}(v_1,v_{k+1})$  if    $k\geq \frac{n}{2}+1$. So, $\epsilon_{C_n}(v_1)=f_{C_n}(v_1,v_{\frac{n}{2}+1})=\frac{n^2+2n+4}{4}$.
Suppose $n$ is odd. Then $f_{C_n}(v_1,v_k)> f_{C_n}(v_1,v_{k+1})$  if  $k\leq \frac{n-1}{2}$ and $f_{C_n}(v_1,v_k)< f_{C_n}(v_1,v_{k+1})$  if    $k\geq \frac{n+3}{2}$. So, $\epsilon_{C_n}(v_1)=f_{C_n}(v_1,v_{\frac{n+1}{2}})=f_{C_n}(v_1,v_{\frac{n+3}{2}})=\frac{n^2+2n+5}{4}$.  Thus,  $\epsilon_{C_n}(v_i)=\lfloor \frac{n^2+2n+5}{4}\rfloor$ for $i=1,2,\ldots,n$ and hence $ES(C_n)=n\lfloor \frac{n^2+2n+5}{4}\rfloor$.
\end{proof}

The next result tells about the number of connected subgraphs of $C_n$ containing a fixed vertex.
\begin{lemma}[\cite{Pandey4}, Lemma 3.1]\label{lem4}
Let $C_n:v_1v_2 \cdots v_nv_1$ be the cycle on $n\geq 3$ vertices. Then $f_{C_n}(v_i)= 2n +\binom{n-1}{2}$ for $i=1,2,\ldots,n.$
\end{lemma}

Let $G\in \mathcal{U}_{n,g}$. Then $G$ can be obtained from the cycle $C_g:v_1v_2\cdots v_gv_1$ and the trees $T_1, T_2,\ldots , T_g$ by identifying a vertex of $T_i$ with the vertex $v_i$ of $C_g$, for $i=1,2,\ldots ,g.$ Here  $|V(T_i)|=n_i+1$ with $n_i\geq 0$ for $i=1,2,\ldots, g$ and $g+n_1+\cdots + n_g=n$. We denote the graph $G$ by $U_{n,g}(T_1,T_2,\ldots ,T_g)$. Note that if $T_i$ is a star graph on $n-g+1$ vertices with $v_i$ as center for some $i$, then $U_{n,g}(T_1,T_2,\ldots ,T_g)=U^p_{n,g}.$

\begin{lemma}\label{lem5}
Let $U_{n,g}(T_1,T_2,\dots,T_g)\in \mathcal{U}_{n,g}$ with $|V(T_i)|=n_i+1$ for $i=1,2,\ldots ,g $. Then $ES(U_{n,g}(T_1,T_2,\ldots ,T_g))\leq ES(U_{n,g}(K_{1,n_1},K_{1,n_2},\dots,K_{1,n_g}))$ where $v_i$ is the center of $K_{1,n_i}$ and the equality happens if and only if $T_i\cong K_{1,n_i}$.
\end{lemma}
\begin{proof}
If $T_i$ is not a star for some $i$, then the vertex $v_i$ is adjacent to a vertex $v$  of $T_i$ with $deg(v)\geq 2$. Identify the vertices $v$ and $v_i$, and add $v$ as a pendant vertex at the identified vertex  $v_{i}$. Repeat this process until $T_i$ becomes $K_{1,n_i}$ with center $v_i$. By Theorem $\ref{thm2}$, the eccentric subgraph index will increase in each step. Continuing this process, we  get $ES(U_{n,g}(T_1,T_2,\dots,T_g))\leq ES(U_{n,g}(K_{1,n_1},K_{1,n_2},\dots,K_{1,n_g}))$ and equality occurs  if and only if $T_i\cong K_{1,n_i}$ for each $1\leq i\leq g$.
\end{proof}
\begin{lemma}\label{lem6}
Let $U^p_{n,g}\in \mathcal{U}_{n,g}$ with all $n-g$ pendant vertices attached to $v_1$ of $C_g:v_1v_2\cdots v_gv_1$. Suppose $v \in V(U^p_{n,g})$. Then the following hold:
\begin{enumerate}
\renewcommand{\labelenumi}{(\roman{enumi})}
\item If $n-g\geq 2$, then  $\epsilon_{U^p_{n,g}}(v)=f_{U^p_{n,g}}(v,y)$, where $y$ is a pendant vertex.
\item If $n-g=1$ and $z$ is the pendant vertex of $U^p_{n,g}$ then 
  $$\epsilon_{U^p_{n,g}}(v)=\left\{
\begin{array}{ll}
f_{U^p_{n,g}}(v,v_{\lfloor \frac{g+2}{2} \rfloor}), &\text{if $v=z$;}\\
f_{U^p_{n,g}}(v,z),  &\text{otherwise.}\\
\end{array}
\right.$$
\end{enumerate}
\end{lemma}
\begin{proof}
\begin{enumerate}
\item[(i)] Let $x,y$ be two pendant vertices of $U^p_{n,g}$ and let $v_j\neq v_1$ be a vertex on $C_g$. Then  $f_{U^p_{n,g}}(x,y)=2^{n-g-2}f_{C_g}(v_1)$. By Lemma \ref{lem4}, we have  $f_{U^p_{n,g}}(x,y)=2^{n-g-2}\left(2g+\binom{g-1}{2}\right)=2^{n-g-2}\left(\frac{g^2+g+2}{2}\right)$. Also, 
\begin{align*}
 f_{U^p_{n,g}}(x,v_j)=2^{n-g-1}f_{C_g}(v_1,v_j)&=2^{n-g-1}\left(2g-1+\binom{g-j}{2}+\binom{j-2}{2}\right)\\
                                                &=2^{n-g-2}(g^2-2gj+2j^2+3g-4j+4). 
\end{align*}
So,
\begin{align*}
  f_{U^p_{n,g}}(x,v_j)-f_{U^p_{n,g}}(x,y)&= 2^{n-g-3}(g^2-4gj+4j^2+5g+6-8j)\\
                                         &=2^{n-g-3}((g-2j+2)^2+g+2)>0,  
\end{align*}

 and also $f_{U^p_{n,g}}(x,v_j)\leq f_{U^p_{n,g}}(x,v_1)$, we have $\epsilon_{U^p_{n,g}}(x)=f_{U^p_{n,g}}(x,y)$.\\

Now consider the vertex $v_1$. We have $f_{U^p_{n,g}}(v_1,y)=2f_{U^p_{n,g}}(x,y)<2f_{U^p_{n,g}}(y,v_j)=f_{U^p_{n,g}}(v_1,v_j)$. Therefore $f_{U^p_{n,g}}(v_1,y)<f_{U^p_{n,g}}(v_1,v_j)$ for  $2\leq j\leq g$. Thus $\epsilon_{U^p_{n,g}}(v_1)=f_{U^p_{n,g}}(v_1,y)$.\\

Let $v_k$ and $v_l$ be two distinct vertices  of $C_g$ with $2\leq k,l\leq g$. Then 
\begin{align}\label{a3}
f_{U^p_{n,g}}(v_k,y)=2^{n-g-1}f_{C_g}(v_k,v_1)=2^{n-g-2}(g^2-2gk+2k^2+3g-4k+4). 
\end{align}

We have $f_{U^p_{n,g}}(v_k,v_l)=f_{C_g}(v_k,v_l,\overline{v_1})+2^{n-g}f_{C_g}(v_k,v_l,v_1).$ First assume $k<l$. Then $f_{C_g}(v_k,v_l,\overline{v_1})=(k-1)(g-l+1)$
and 
\begin{align*}
f_{C_g}(v_k,v_l,v_1)&=f_{C_g}(v_k,\ldots,v_1,v_g,v_{g-1},\ldots,v_l)+f_{C_g}(v_k,\ldots,v_l)- f_{C_g}(v_k,v_l,\overline{v_1})-1\\
                    &=2(l-k)+\binom{l-k-1}{2}+2(g-l+k)+\binom{g-l+k-1}{2}-(k-1)(g-l+1)-1\\
                    &=3g-l-k-gk+kl+\binom{l-k-1}{2}+\binom{g-l+k-1}{2}\\
                    &=3g-l-k-gk+kl+\frac{g^2+2l^2+2k^2-4kl+2gk-2gl-3g+4}{2}.                                                                                                                                                                                                                                                                                                                                                                                                                                                                                                                                                                                                                                                                                                                                                                                                                                                                                                           
\end{align*}
Thus, 
\begin{align}\label{a4}
f_{U^p_{n,g}}(v_k,v_l)=(k-1)(g-l+1)+2^{n-g-1}(g^2+2l^2+2k^2+3g+4-2kl-2gl-2l-2k).
\end{align}
 So by (\ref{a3}) and (\ref{a4}), we get $f_{U^p_{n,g}}(v_k,v_l)- f_{U^p_{n,g}}(v_k,y)=(k-1)(g-l+1)+2^{n-g-2}((g-2l+k+1)^2+(k-1^2)+g+2)>0.$\\   

Similarly, it can be checked that for $k>l$, $f_{U^p_{n,g}}(v_k,v_l)- f_{U^p_{n,g}}(v_k,y)>0$. Since $f_{U^p_{n,g}}(v_k,v_1)> f_{U^p_{n,g}}(v_k,y)$, we have 
 $\epsilon_{U^p_{n,g}}(v_k)=f_{U^p_{n,g}}(v_k,y)$ for $2\leq k\leq g$. This proves the result.

\item[(ii)]  Let $z$  be the the unique pendant vertex adjacent to $v_1$. Using a similar argument as in the first part, we can prove that for $1\leq i\leq g$,  $\epsilon_{U^p_{n,g}}(v_i)=f_{U^p_{n,g}}(v_i,z)$. Since $f_{U^p_{n,g}}(z,v_j)=f_{C_g}(v_1,v_j)$ and by Proposition $\ref{prop1}$ we know that $f_{C_g}(v_1,v_j)\geq f_{C_g}(v_1,v_{\lfloor \frac{g+2}{2}\rfloor})$ and this proves the result.
\end{enumerate}   
\end{proof}

\begin{theorem}\label{thm3}
Let $G\in \mathcal{U}_{n,g}$. Then $ES(G)\leq ES(U^p_{n,g})$ and the equality happens if and only if $G\cong U^p_{n,g}.$
\end{theorem}
\begin{proof}
The result is obvious if $n-g= 0$ or $1$. So, assume that $n-g\geq 2$. Let $G=U_{n,g}(T_1,T_2,\dots,T_g)\in \mathcal{U}_{n,g}$ with $|V(T_i)|=n_i+1$($n_i\geq 0$) for $i=1,2,\ldots ,g $. Then by by Lemma $\ref{lem5}$,  we have $\epsilon(G))\leq\epsilon(U_{n,g}(K_{1,n_1},K_{1,n_2},\ldots,K_{1,n_g}))$ and the equality occurs  if and only if $T_i\cong K_{1,n_i}$ with $v_i$ is the center of $K_{1,n_i}$ for $1\leq i\leq g$.

W.l.o.g., take $n_1\geq n_2\geq \cdots \geq n_g$. We have  $U_{n,g}(K_{1,n_1},K_{1,n_2},\ldots,K_{1,n_g}) \cong U^p_{n,g}$ if and only if $n_1=n-g$. Suppose $n_1\neq n-g$. Then there exists a vertex $v_k\neq v_1$ such that  $deg(v_k)\geq 3$.  We can obtain $U^p_{n,g}$ from $U_{n,g}(K_{1,n_1},K_{1,n_2},\dots,K_{1,n_g})$ by deleting all the pendant vertices of $U_{n,g}(K_{1,n_1},K_{1,n_2},\dots,K_{1,n_g})$ and adding them as pendant vertices at $v_1$. It is enough to show that $\epsilon(U_{n,g}(K_{1,n_1},K_{1,n_2},\dots,K_{1,n_g}))<\epsilon(U^p_{n,g})$.

By Lemma $\ref{lem6}$, for any $v\in V(U^p_{n,g})$, we have  $\epsilon_{U^p_{n,g}}(v)=f_{U^p_{n,g}}(v,x)$, where $x$ is a pendant vertex of $U^p_{n,g}$. Let $y$ be a pendant vertex at $v_1$ in $U_{n,g}(K_{1,n_1},K_{1,n_2},\dots,K_{1,n_g})$. Then

\begin{align*}
  \epsilon_{U^p_{n,g}}(v_1)= f_{U^p_{n,g}}(v_1,y)=2^{(n-g-1)}f_{C_g}(v_1)=2^{(n-g-1)}f_{C_g}(v_1,v_k)+2^{(n-g-1)}f_{C_g}(v_1,\overline{v_k}). 
\end{align*}
 Since $$f_{U_{n,g}(K_{1,n_1},K_{1,n_2},\dots,K_{1,n_g})}(v_1,\overline{v_k},y)<2^{(n-g-1)}f_{C_g}(v_1,\overline{v_k})$$
  and $f_{U_{n,g}(K_{1,n_1},K_{1,n_2},\dots,K_{1,n_g})}(v_1,v_k,y)\leq 2^{(n-g-1)}f_{C_g}(v_1,v_k)$ (with equality happens if and only if $n_1+n_k=n-g$), 
\begin{align*}
  \epsilon_{{U_{n,g}(K_{1,n_1},K_{1,n_2},\dots,K_{1,n_g})}}(v_1)\leq  f_{U_{n,g}(K_{1,n_1},K_{1,n_2},\dots,K_{1,n_g})}(v_1,y)< f_{U^p_{n,g}}(v_1,y)=\epsilon_{U^p_{n,g}}(v_1).
\end{align*}
For $2\leq i \leq g$, we have $f_{U^p_{n,g}}(v_i,y)=2^{(n-g-1)}f_{C_g}(v_i,v_1)$ and so
 $$ \epsilon_{U_{n,g}(K_{1,n_1},K_{1,n_2},\dots,K_{1,n_g})}(v_i)\leq f_{U_{n,g}(K_{1,n_1},K_{1,n_2},\dots,K_{1,n_g})}(v_i,y)\leq f_{U^p_{n,g}}(v_i,y)=\epsilon_{U^p_{n,g}}(v_i)$$ where second inequality becomes equality if and only if $n_1+n_i=n-g$.
 
Now let $z$ be a pendant vertex  of $U_{n,g}(K_{1,n_1},K_{1,n_2},\dots,K_{1,n_g})$ adjacent to $v_j$ for some $j$. Then $\epsilon_{U^p_{n,g}}(z)=f_{U^p_{n,g}}(z,x)=2^{n-g-2}f_{C_g}(v_1)$ where $x\neq z$ is a pendant vertex of $U^p_{n,g}$. So, $$\epsilon_{U_{n,g}(K_{1,n_1},K_{1,n_2},\dots,K_{1,n_g})}(z)\leq f_{U_{n,g}(K_{1,n_1},K_{1,n_2},\dots,K_{1,n_g})}(z,x)< f_{U^p_{n,g}}(z,x)=\epsilon_{U^p_{n,g}}(z).$$ 

Thus $\epsilon_{U_{n,g}(K_{1,n_1},K_{1,n_2},\dots,K_{1,n_g})}(v)\leq \epsilon_{U^p_{n,g}}(v)$ for all vertex $v$, with strict inequality atleast at one vertex. Hence, $\epsilon(U_{n,g}(K_{1,n_1},K_{1,n_2},\dots,K_{1,n_g}))<\epsilon(U^p_{n,g})$ and this proves the result. 
\end{proof}

\begin{lemma}\label{lem7}
Let $n\geq 5$ and $4\leq g < n$. Then $ES(U^p_{n,g})<ES(U^p_{n,{g-1}})$. 
\end{lemma}
\begin{proof}
 Let  $C_g:v_1v_2\cdots v_gv_1$ be the cycle in $U^p_{n,g}$ with  all $n-g$ pendant vertices attached to $v_1$. Obtain $U^p_{n,g-1}$ form $U^p_{n,g}$ by identifying the vertices $v_g$ and $v_1$ as $v_1$ and adding $v_g$ as a pendant vertex at $v_1$. Let $y$ be a pendant vertex at $v_1$ in both $U^p_{n,g}$ and $U^p_{n,g-1}$. Then 
 
 $$f_{U^p_{n,g}}(v_1,y)=2^{n-g-1} \left(2g+\binom{g-1}{2}\right)\;\; \mbox{and} \;\; f_{U^p_{n,{g-1}}}(v_1,y)=2^{n-g} \left(2g-2+\binom{g-2}{2}\right),$$ also 

  $$f_{U^p_{n,g}}(v_g,y)=2^{n-g-1}\left(2g-1+\binom{g-2}{2}\right)\;\; \mbox{and} \;\; f_{U^p_{n,{g-1}}}(v_g,y)=2^{n-g-1}\left(2(g-1)+\binom{g-2}{2}\right).$$

So, $f_{U^p_{n,g}}(v_1,y)+f_{U^p_{n,g}}(v_g,y)=2^{n-g-1}\left(4g-1+\binom{g-1}{2}+\binom{g-2}{2}\right)=2^{n-g-1}(g^2+3)$ and 
\begin{align*}
f_{U^p_{n,{g-1}}}(v_1,y)+f_{U^p_{n,{g-1}}}(v_g,y)&=2^{n-g-1}\left(6(g-1)+3\binom{g-2}{2}\right)\\
                                                 &>2^{n-g-1}(6(g-1)+(g-2)(g-3))\\
                                                 &=2^{n-g-1}(g^2+g).
\end{align*}
By Lemma $\ref{lem6}$, we have $\epsilon_{U^p_{n,{g-1}}}(v_1)=f_{U^p_{n,{g-1}}}(v_1,y)$ and $\epsilon_{U^p_{n,{g-1}}}(v_g)=f_{U^p_{n,{g-1}}}(v_g,y)$ and thus
\begin{align*} 
\epsilon_{U^p_{n,g}}(v_1)+\epsilon_{U^p_{n,g}}(v_g)&\leq f_{U^p_{n,g}}(v_1,y)+f_{U^p_{n,g}}(v_g,y)\\
                                                   &=2^{n-g-1}(g^2+3)\\
                                                   &< 2^{n-g-1}(g^2+g)\\
                                                   &<f_{U^p_{n,{g-1}}}(v_1,y)+f_{U^p_{n,{g-1}}}(v_g,y)\\
                                                   &=\epsilon_{U^p_{n,{g-1}}}(v_1)+\epsilon_{U^p_{n,{g-1}}}(v_g).
\end{align*}

For $2\leq j\leq g-1$, we have 
$$ f_{U^p_{n,g}}(v_j,y)=2^{n-g-1}\left(2g-1+\binom{g-j}{2}+\binom{j-2}{2}\right)$$  and 
$$f_{U^p_{n,{g-1}}}(v_j,y)=2^{n-g}\left(2(g-1)-1+\binom{g-1-j}{2}+\binom{j-2}{2}\right).$$ So,
\begin{align*}
    f_{U^p_{n,{g-1}}}(v_j,y)-f_{U^p_{n,g}}(v_j,y)&=2^{n-g-1}\left(2g-5+2\binom{g-1-j}{2}-\binom{g-j}{2}+\binom{j-2}{2}\right)\\
                                                 &=2^{n-g-1}\left(2g-5+\binom{g-j-1}{2}-\binom{g-j}{2}+\binom{g-j-1}{2}+\binom{j-2}{2}\right)\\
                                                 &=2^{n-g-1}\left(2g-5-(g-j-1)+\binom{g-j-1}{2}+\binom{j-2}{2}\right)\\
                                                 &=2^{n-g-1}\left(g+j-4+\binom{g-j-1}{2}+\binom{j-2}{2}\right)\\
                                                 &>0\;\; \mbox{(since $g\geq 4$  and $j\geq 2$)}.
\end{align*}
Thus, $\epsilon_{U^p_{n,{g}}}(v_j)\leq  f_{U^p_{n,{g}}}(v_j,y)<f_{U^p_{n,{g-1}}}(v_j,y)=\epsilon_{U^p_{n,{g-1}}}(v_j)$ for $2\leq j \leq g-1$.\\

Let $x$ be a pendant vertex of $U^p_{n,{g}}$. If $n-g\geq 2$ then by Lemma $\ref{lem6}$, we have $\epsilon_{U^p_{n,{g}}}(x)=f_{U^p_{n,{g}}}(x,y)$ for some pendant vertex $y\neq x$. So, $f_{U^p_{n,g}}(x,y)=2^{n-g-2} (2g+\binom{g-1}{2})$ and $f_{U^p_{n,{g-1}}}(x,y)=2^{n-g-1} (2g-2+\binom{g-2}{2})$. Thus 
\begin{align*}
   f_{U^p_{n,{g-1}}}(x,y)-f_{U^p_{n,g}}(x,y)&=2^{n-g-2}\left(2g-4+2\binom{g-2}{2}-\binom{g-1}{2}\right)\\
   &=2^{n-g-2}\left(2g-4+\binom{g-2}{2}-\binom{g-1}{2}+\binom{g-2}{2}\right)\\
   &=2^{n-g-2}\left(2g-4-(g-2)+\binom{g-2}{2}\right)\\
   &=2^{n-g-2}\left(g-2+\binom{g-2}{2}\right)>0  \;\; \mbox{(since $g\geq 4$)}.
\end{align*}
So, $\epsilon_{U^p_{n,{g}}}(x)<\epsilon_{U^p_{n,{g-1}}}(x)$ if $n-g\geq 2$.

If $n-g=1$ then $x=y$. By Lemma $\ref{lem6}$, we have $\epsilon_{U^p_{n,{g}}}(x)=f_{U^p_{n,g}}(x,v_{\lfloor \frac{g+2}{2} \rfloor})$ and  $\epsilon_{U^p_{n,{g-1}}}(x)=f_{U^p_{n,g-1}}(x,v_g)=2(g-1)+\binom{g-2}{2}$. If $g$ is even  
\begin{align*}
 f_{U^p_{n,g-1}}(x,v_g)-f_{U^p_{n,g}}(x,v_{\lfloor \frac{g+2}{2} \rfloor})= \frac{g^2-g+2}{2}-\frac{g^2+2g+4}{4}=\frac{g^2-4g}{4}\geq 0,\; \mbox{(since $g\geq 4$)},
\end{align*}
 and if $g$ is odd 
\begin{align*}
 f_{U^p_{n,g-1}}(x,v_g)-f_{U^p_{n,g}}(x,v_{\lfloor \frac{g+2}{2} \rfloor})= \frac{g^2-g+2}{2}-\frac{g^2+2g+5}{4}=\frac{g^2-4g-1}{4}>0,\; \mbox{(since $g\geq 5$)}.
\end{align*}
Thus $\epsilon_{U^p_{n,{g}}}(x)\leq\epsilon_{U^p_{n,{g-1}}}(x)$ and this proves the result.

\end{proof}

\begin{lemma}\label{lem8}
Let $n\geq 4$. Then  $ES(U^p_{n,3})=
      \begin{cases}
           22 , & if\  n=4;\\
           2^{n-5}(13+7n) , & if \ n\geq 5.
         \end{cases}$   
\end{lemma}
\begin{proof}
For $n=4$, there is single pendant vertex $x$ at $v_1$. Then $\epsilon_{U^p_{4,3}}(v_1)=7$, $\epsilon_{U^p_{4,3}}(v_2)=\epsilon_{U^p_{4,3}}(v_3)=5$ and $\epsilon_{U^p_{4,3}}(x)=f_{U^p_{4,3}}(x,v_2)=5$. So, $ES({U^p_{n,3}})=22$. For $n\geq 5$, $U^p_{n,3}$  has $n-3\geq 2$ pendant vertices attached to $v_1$. Then $\epsilon_{U^p_{n,3}}(v_1)=7\times2^{n-4}$ and $\epsilon_{U^p_{n,3}}(v_2)=\epsilon_{U^p_{n,3}}(v_3)=5\times2^{n-4}$. Also, for any  pendant vertex $x$ at $v_1$,   $\epsilon_{U^p_{n,3}}(x)=7\times2^{n-5}$. Thus, $ES({U^p_{n,3}})=7\times2^{n-4}+2\times5\times2^{n-4}+(n-3)\times7\times2^{n-5}=2^{n-5}(13+7n)$.
\end{proof}

\begin{remark}\label{rem1}
From  Proposition  $\ref{prop1}$, we have, $ES(C_4)=28$  and $ES(C_5)=50$. Hence,  $ES(U^p_{4,3})<ES(C_4)$ and also $ES(U^p_{5,3})<ES(C_5)$. 
\end{remark}
\begin{theorem}\label{thm4}
 Let $G$ be a unicyclic graph on $n\geq 6$ vertices. Then $ES(G)\leq ES(U^p_{n,3})$ and the equality happens if and only if $G\cong U^p_{n,3}.$  
\end{theorem}
\begin{proof}
From Theorem $\ref{thm3}$ and Lemma $\ref{lem7}$, it follows that either $U^p_{n,3}$ or $C_n$ maximizes the eccentric subgraph index  among all  unicyclic graphs on $n$ vertices. From  Proposition  $\ref{prop1}$ and  Lemma $\ref{lem8}$, we have $ES(C_n)=n\lfloor \frac{n^2+2n+5}{4}\rfloor$ and $ES({U^p_{n,3}})=2^{n-5}(13+7n)$. So,
\begin{align*}
  \epsilon(U^p_{n,3})-\epsilon(C_n)&=2^{n-5}(13+7n)-n\lfloor \frac{n^2+2n+5}{4}\rfloor\\
  &\geq 2^{n-5}(13+7n)- \frac{n^3+2n^2+5n}{4}\\
  &=\frac{2^{n-3}(13+7n)-(n^3+2n^2+5n)}{4}\\
  &=\frac{n(5\times2^{n-3}-n^2)+2(5\times2^{n-3}-n^2)+n(2\times2^{n-3}-5)+3\times 2^{n-3}}{4}\\
  &=\frac{(5\times2^{n-3}-n^2)(n+2)+n(2^{n-2}-5)+3\times 2^{n-3}}{4}>0,
\end{align*}
as $5\times2^{n-3}-n^2>0$ and $2^{n-2}-5>0$ for $n\geq 6$. This proves the result.
\end{proof}

\section{Trees with fixed diameter}\label{tfd}

Let $T$ be a tree. The eccentric diameter $\epsilon^d(T)$ of $T$ is defined as $\epsilon^d(T)=\min\{f_T(u,v): u, v\in V(T)\}$. The following useful results on eccentric subtree number and eccentric diameter are due to Zhang et al.

\begin{theorem}[\cite{Zhang}, Proposition 2.1, Theorem 2.3]\label{zha1}
Let $T$ be a tree. Then the following hold:
\begin{enumerate}
\item[(i)] For $v\in V(T),$ if $\epsilon(v)= f_T(v,u)$ then $u$ is a pendant vertex;
\item[(ii)] If $\epsilon^d(T)=f_T(u,v)$ then both $u$ and $v$ are pendant vertices;
\item[(iii)] If $\epsilon^d(T)=f_T(u,v)$ for some pendant vertices $u$ and $v$ then $\epsilon(w)=\min\{f_T(w,u), f_T(w,v)\}$ for any $w\in V(T)$.
\end{enumerate}
\end{theorem}

Let $\mathcal{T}^n_{d}$ be the set of all trees of order $n$ with diameter $d$. Note that $K_2$ is the only tree with diameter $1$ and for $n\geq 3$, $K_{1,{n-1}}$ is the only tree with diameter $2$. Also, $P_n$ is the only tree on $n$ vertices with diameter $n-1$. Therefore, we consider $3\leq d\leq n-2$. Let $C^n_d$ be the set of all caterpillar trees on $n$ vertices with diameter $d$. We can form any caterpillar tree in the following way. For $1\leq i \leq d-1$, take $n_i\geq 0$ with $n_1+n_2+\cdots +n_{d-1}=n-d-1$. Consider the path $P_{d+1}:v_0v_1\cdots v_d$ and form the caterpillar tree by attaching $n_i$  pendant vertices at $v_i$ for $1\leq i\leq d-1$.

\begin{lemma}\label{lem9}
Let $T\in\mathcal{T}^n_{d}$. If $T\notin C^n_d$ then there exist a $T'\in C^n_d$ such that $ES(T)<ES(T')$.
\end{lemma}
\begin{proof}
Let  $P_{d+1}:v_0v_1\cdots v_d$ be a path in $T$. For $1\leq i\leq d-1$, Let $X_i$ be the component of  $T-v_{i-1}v_i-v_iv_{i+1}$ containing $v_i$. If $T\notin C^n_d$  then there exists $j$, where $1\leq j\leq d-1$ such that $X_j$ is not a star with center at $v_j$. So, $v_j$ is adjacent to a vertex $v\in V(X_j)$ such that $deg(v)\geq 2$. Identify $v$ and $v_j$, and add $v$ as a pendant vertex at the identified vertex  $v_{j}$. Repeat this process until $X_j$ becomes  a star with center $v_j$. Continue this process till all $X_i$ are stars with center $v_i$. Call the new tree as $T'$. Clearly $T'\in C^n_d$. By Theorem $\ref{thm2}$, in each step  the eccentric subgraph index will increase and hence $ES(T)<ES(T')$.
\end{proof}

\begin{lemma}\label{lem10}
Let $T\in C^n_d$ and let $P_{d+1}:v_0v_1\cdots v_d$ be a path in $T$. Then for  $v\in V(T)$,  $\epsilon_T(v)=min\{f_T(v,v_0), f_T(v,v_d)\}$.    
\end{lemma}
\begin{proof}
Consider the path $P_{d+1}:v_0v_1\cdots v_d$  in $T$. Then $v_0$ and $v_d$ are pendant vertices and  $f_T(v_0,v_d)=2^{n-d-1}$.
Let $x,y\in V(T)$ be two pendant vertices. Suppose $x$ and $y$ are adjacent to $v_i$ and $v_j$, respectively where $1\leq i
<j\leq d-1$. Let $X,Y$ and $Z$ be the components of $T-v_{i-1}v_{i}-v_jv_{j+1}$, with $v_{i-1}\in V(X)$, $v_i,v_j\in  V(Y)$ and $v_{j+1}\in V(Z)$.  We have $f_Y(x,y)=2^{n_i+\dots+n_j-2}$, $f_X(v_{i-1})=2^{n_{i-1}}+2^{{n_{i-1}+n_{i-2}}}+\dots+2^{n_{i-1}+\dots+n_1+1}$ and $f_Z(v_{j+1})=2^{n_{j+1}}+2^{n_{j+1}+n_{j+2}}+\dots+2^{n_{j+1}+\dots+n_{d-1}+1}$. Then
\begin{align*}
f_T(x,y)&=f_Y(x,y)(1+f_X(v_{i-1})+f_Z(v_{j+1})+f_X(v_{i-1})f_Z(v_{j+1}))\geq 2^{n_1+\dots+n_{d-1}}.
\end{align*}
So $f_T(v_0,v_d)\leq f_T(x,y)$. Hence, $\epsilon^d(T)=f_T(v_0,v_d)$ and  by Theorem $\ref{zha1}$(iii), we have  $\epsilon_T(v)=min\{f_T(v,v_0), f_T(v,v_d)\}$.
\end{proof}

Suppose $d$ is even. Consider the path $P_{d+1}:v_0v_1\cdots v_d$ and form the tree $T^e_{n,d}$ by attaching $n-d-1$ pendant vertices to $v_{\frac{d}{2}}$. Clearly  $T^e_{n,d} \in C^n_d$.
\begin{lemma}\label{lem11}
Suppose $d=2k$. Then
  \begin{enumerate}
  \renewcommand{\labelenumi}{(\roman{enumi})}
\item $\epsilon_{T^e_{n,d}}(v_i)=
      \begin{cases}
           f_{T^e_{n,d}}(v_i,v_d),& \;\mbox{if}\;\;  0\leq i\leq k;\\
           f_{T^e_{n,d}}(v_i,v_0),& \;\mbox{if}\;\; k\leq i\leq d.
         \end{cases}$  
\item $\epsilon_{T^e_{n,d}}(x)=f_{T^e_{n,d}}(x,v_0)=f_{T^e_{n,d}}(x,v_d)$, where $x$ is a pendant vertex attached to $v_k$.
  \end{enumerate}    
\end{lemma}
\begin{proof}
\begin{enumerate}
\item[(i)] For $0\leq i\leq k$, we have $f_{T^e_{n,d}}(v_i,v_0)=(k-i)+(k+1) 2^{n-d-1}$  and $f_{T^e_{n,d}}(v_i,v_d)=(i+1)2^{n-d-1}$. Then $f_{T^e_{n,d}}(v_i,v_0)-f_{T^e_{n,d}}(v_i,v_d)=(k-i)(1+2^{n-d-1})\geq 0$  and equality holds if $i=k$. So, by Lemma $\ref{lem10}$, $\epsilon_{T^e_{n,d}}(v_i)=f_{T^e_{n,d}}(v_i,v_d)$. Similarly, we can show that for $k\leq i\leq d$, $\epsilon_{T^e_{n,d}}(v_i)=f_{T^e_{n,d}}(v_i,v_0)$.
\item[(ii)] We have $f_{T^e_{n,d}}(x,v_0)=2^{n-d-2}(k+1)=f_{T^e_{n,d}}(x,v_d).$ So, by Lemma $\ref{lem10}$, $\epsilon_{T^e_{n,d}}(x)=f_{T^e_{n,d}}(x,v_0)=f_{T^e_{n,d}}(x,v_d)$.
\end{enumerate}   
\end{proof}
\begin{remark}\label{rem3}
Let $3\leq d\leq n-2$. Then the eccentric subgraph index of  $T^e_{n,d}$ is 
 \begin{align*}
   ES(T^e_{n,d})&=2\displaystyle\sum_{i=0}^{k-1} (i+1)2^{n-d-1}+(k+1)2^{n-d-1}+(n-d-1)(k+1)2^{n-d-2}\\
   &=k(k+1)2^{n-d-1}+(k+1)2^{n-d-1}+(n-d-1)(k+1)2^{n-d-2}\\
   &=(k+1)^22^{n-d-1}+(n-d-1)(k+1)2^{n-d-2}\\
   &=(n+1)(\frac{d+2}{2})2^{n-d-2}.
 \end{align*}
\end{remark}
Suppose $d$ is odd. Consider the path $P_{d+1}:v_0v_1\cdots v_d$ and form the tree $T^{o(r,s)}_{n,d}$ by attaching $r\geq 0$ pendant vertices to $v_{\frac{d-1}{2}}$ and $s\geq 0$ pendant vertices to $v_{\frac{d+1}{2}}$ where $r+s=n-d-1$. Clearly  $T^{o(r,s)}_{n,d} \in C^n_d$.
\begin{lemma}\label{lem12}
Let $d=2k-1$ and let $T=T^{o(r,s)}_{n,d}$. Then
\begin{enumerate}
  \renewcommand{\labelenumi}{(\roman{enumi})}
\item $\epsilon_T(v_i)=
      \begin{cases}
           f_T(v_i,v_d),& \;\mbox{if}\;\;  0\leq i\leq k-1;\\
           f_T(v_i,v_0),& \;\mbox{if}\;\; k\leq i\leq d.
         \end{cases}$  
\item $\epsilon_T(x)=
      \begin{cases}
           f_T(x,v_d),& \;\mbox{if $x$ is a pendant vertex attached to $v_{k-1}$;}\\
           f_T(x,v_0),& \;\mbox{if $x$ is a pendant vertex attached to $v_{k}$.}
         \end{cases}$
\end{enumerate}
\end{lemma}
\begin{proof}
\begin{enumerate}
\item [(i)] For $0\leq i\leq k-1$, we have $f_T(v_i,v_0)=(k-1-i)+2^r+k2^{r+s}$ and $f_T(v_i,v_d)=(i+1)2^{r+s}$. Then, $f_T(v_i,v_0)-f_T(v_i,v_d)=2^r+(k-i-1)(1+2^{r+s})>0$. So, by Lemma $\ref{lem10}$, $\epsilon_T(v_i)=f_T(v_i,v_d)$. Similarly, it can be checked that for $k\leq i\leq d$, $\epsilon_T(v_i)=f_T(v_i,v_0)$.

\item[(ii)] Let $x$ be a pendant vertex attached to $v_{k-1}$. So, $ f_T(x,v_0)=2^{r-1}+k2^{r+s-1}$ and $f_T(x,v_d)=k2^{r+s-1}$. Then  $f_T(x,v_0)-f_T(x,v_d)=2^{r-1}>0$. So, by Lemma $\ref{lem10}$, $\epsilon_T(x)=f_T(x,v_d)$. Similarly, it can be checked that $\epsilon_T(x)=f_T(x,v_0)$ if $x$ be a pendant vertex attached to $v_k$. 
\end{enumerate}    
\end{proof}
\begin{remark}\label{rem4}
Let $3\leq d\leq n-2$. Then the eccentric subgraph index of  $T^{o(r,s)}_{n,d}$ is 
 \begin{align*}
   ES(T^{o(r,s)}_{n,d})&=\displaystyle\sum_{i=0}^{k-1} (i+1)2^{r+s}  +\displaystyle\sum_{i=k}^{2k-1}2^{r+s}(2k-i)+k(r+s)2^{r+s-1}\\
   &=k(k+1)2^{r+s}+k(r+s)2^{r+s-1}\\
   &=k(k+1)2^{n-d-1}+k(n-d-1)2^{n-d-2}\;\; \mbox{($r+s=n-d-1)$}\\
   &=(\frac{d+1}{2})(n+2)2^{n-d-2}.
 \end{align*}
Since $ES(T^{o(r,s)}_{n,d})$  dose not depend on $r$ or $s$, hence $ES(T^{o(r,s)}_{n,d})$ is fixed for any $r,s\geq 0$ with $r+s=n-d-1.$
\end{remark}
\begin{theorem}\label{thm5}
Let $3\leq d\leq n-2$ and let $T\in \mathcal{T}^n_{d}$. 
  \begin{enumerate}
  \renewcommand{\labelenumi}{(\roman{enumi})}
      \item Suppose $d$ is even. Then $ES(T)\leq (n+1)(\frac{d+2}{2})2^{n-d-2}$  with equality happens if and only if $T\cong T^e_{n,d}$.
      \item Suppose $d$ is odd. Then $ES(T)\leq (\frac{d+1}{2})(n+2)2^{n-d-2}$ with equality happens if and only if $T\cong T^{o(r,s)}_{n,d}$ for some $r,s\geq 0$ with $r+s=n-d-1$.
  \end{enumerate}
\end{theorem}
\begin{proof}
\begin{enumerate}
\item [(i)] Let $d$ be even. Suppose $T$ is not isomorphic to $T^e_{n,d}$. We have two cases:\\
{\bf Case-I:} $T\in C^n_d$\\
Let $P_{d+1}:v_0v_1\cdots v_d$ be a path in $T$. W.l.o. g. we can take $deg_{T}(v_i)\geq 3$ for some $i$ where $1\leq i \leq \frac{d}{2} -1$ as $T$ is not isomorphic to $T^e_{n,d}$. We can get $T^e_{n,d}$  from $T$ by deleting all the pendant vertices of $T$ except $v_0$ and $v_d$ and attaching them as pendant vertices at $v_{\frac{d}{2}}$. Then $\epsilon_T(v_{\frac{d}{2}})\leq f_T(v_{\frac{d}{2}},v_d)< (k+1)2^{n-d-1}=\epsilon_{T^e_{n,d}}(v_{\frac{d}{2}})$. It can be checked that for any $x\in V(T)\setminus v_{\frac{d}{2}}$, $\epsilon_T(x)\leq \epsilon_{T^e_{n,d}}(x)$. Thus $ES(T)<ES(T^e_{n,d})$.\\
{\bf Case-II:} $T\notin C^n_d$\\
By Lemma \ref{lem9}, there exists a $T'\in C^n_d$ such that $ES(T)<ES(T')$ and then by Case-I, we have  $ES(T)<ES(T')\leq ES(T^e_{n,d})$.
Now the result follows from Remark \ref{rem3}.
\item [(ii)]  Let $d$ be odd. Suppose $T$ is not isomorphic to $T^{o(r,s)}_{n,d}$ for any $r,s\geq 0$ with $r+s=n-d-1$. Then similarly to (i), we can show that  $ES(T)<ES(T^{o(r,s)}_{n,d})$ for some $r,s\geq 0$ with $r+s=n-d-1$. Now the result follows from Remark \ref{rem4}.
\end{enumerate}

\end{proof}

\section{Graphs with fixed number of pendant vertices}\label{gfp}

Let $\mathcal{G}^k_n$ be the set of all graphs of order $n$ with $k$ pendant vertices. Note that $K_2$ is the only tree with all pendant vertices. For $n\geq 3$, $K_{1,{n-1}}$ is the only graph with $n-1$ pendant vertices. For $n\geq 4$, if $G\in \mathcal{G}^{n-2}_n$ then  $G\cong T^{o(r,s)}_{n,3}$ for some $r,s\geq 0$ with $r+s=n-4$ and by Theorem \ref{thm5}, $ES(G)=(n+2)2^{n-d-1}$. So, we consider $0\leq k\leq n-3$ and $n\geq 4$. 

Take $m_1\geq m_2\geq \cdots \geq m_{n-k}\geq 0$ with $m_1+m_2+\cdots +m_{n-k}=k$. Consider the complete graph $K_{n-k}$ with vertices $v_1, v_2,\ldots, v_{n-k}$ and form a new graph by attaching $m_i$  pendant vertices at $v_i$ for $1\leq i\leq n-k$. We denote this graph by $K(m_1,m_2,\ldots,m_{n-k})$. For $m_1=k$, we denote it by $K^k_{n-k}$. Let $G^k_n(S)=\{ K(m_1,m_2,\ldots,m_{n-k})| m_1\geq m_2\geq \cdots \geq m_{n-k}\geq 0 \; \mbox{and} \; m_1+m_2+\cdots +m_{n-k}=k \}$.

\begin{lemma}\label{lem13}
Let $G\in \mathcal{G}^k_n$. If $G\notin G^k_n(S)$ then there exist a $G'\in G^k_n(S)$ such that $ES(G)<ES(G')$.
\end{lemma}
\begin{proof}
Let $v_1,v_2,\ldots,v_{n-k}$ be the non-pendant vertices of $G$. Since $G\notin G^k_n(S)$, there exist $v_i$ and $v_j$ such that $v_i$ is not adjacent to $v_j$ in $G$. Obtain a new graph $G'$ from $G$ by joining  all such non-adjacent pair of vertices from the set $\{v_1,v_2,\ldots,v_{n-k}\}$. Then by Lemma \ref{edge}, $ES(G)<ES(G')$.
\end{proof}
For $k\in \{0,1\}$, it follows from Lemma \ref{lem13} that $K^k_{n-k}$ uniquely maximizes the eccentric subgraph index over $\mathcal{G}^k_n$. So, we consider $2\leq k\leq n-3$
\begin{lemma}\label{lem14}
Let $2\leq k\leq n-3$ and let $K^k_{n-k}$ be the graph defined as above. 
\begin{enumerate}
\renewcommand{\labelenumi}{(\roman{enumi})}
\item Then $\epsilon(v_1)=min\{f_{K^k_{n-k}}(v_1,v_j), f_{K^k_{n-k}}(v_1,x)\}$ where $j\neq 1$ and for $2\leq i\leq n-k$, $\epsilon(v_i)=f_{K^k_{n-k}}(v_i,x)$ where $x$ is a pendant vertex.     
\item Let $x$ and $y$ be  pendant vertices at $v_1$. Then $\epsilon(x)=\min\{f_{K^k_{n-k}}(x,y),f_{K^k_{n-k}}(x,v_j)\}$ where $j\neq 1$.
\end{enumerate}
\begin{proof}
\begin{enumerate}
\item [(i)] 
Let $2\leq r,s\leq n-k$. We have $f_{K^k_{n-k}}(v_1,v_r)=2^kf_{K_{n-k}}(v_1,v_r)=2^k f_{K_{n-k}}(v_1,v_s)$. Let $x,y$ be pendant vertices of $K^k_{n-k}$. Then $f_{K^k_{n-k}}(v_1,x) =f_{K^k_{n-k}}(v_1,y)=2^{k-1}f_{K_{k-n}}(v_1)$. So, $\epsilon_{K^k_{n-k}}(v_1)=min\{f_{K^k_{n-k}}(v_1,v_j),f_{K^k_{n-k}}(v_1,x)\}$ where $j\neq 1$.\\

Now let $i\neq 1$ and $r\neq i$. Then  $f_{K^k_{n-k}}(v_i,v_r)=f_{K_{n-k}}(v_i,v_r,\overline{v_1})+2^kf_{K_{n-k}}(v_i,v_r,v_1)$ and $f_{K^k_{n-k}}(v_i,x)=2^{k-1}f_{K_{n-k}}(v_i,v_1)$ where $x$ is a pendant vertex. We have
       \begin{align*}
        f_{K^k_{n-k}}(v_i,v_r)-  f_{K^k_{n-k}}(v_i,x)&=f_{K_{n-k}}(v_i,v_r,\overline{v_1})+2^kf_{K_{n-k}}(v_i,v_r,v_1)-2^{k-1}f_{K_{n-k}}(v_i,v_1)\\
        &=f_{K_{n-k}}(v_i,v_r,\overline{v_1})+2^{k-1}f_{K_{n-k}}(v_i,v_1,v_r)-2^{k-1}f_{K_{n-k}}(v_i,v_1,\overline{v_r})\\
        &>0
       \end{align*}
as we can obtain a connected subgraph of $K_{n-k}$ that contain $v_i,v_1$ and $v_r$ by attaching $v_r$ to $v_1$ in a connected subgraph of $K_{n-k}$ that contain $v_i,v_1$ but not $v_r$  and thus $f_{K_{k-n}}(v_i,v_1,\overline{v_r})\leq f_{K_{k-n}}(v_i,v_1,v_r)$. Hence $\epsilon_{K^k_{n-k}}(v_i)=f_{K^k_{n-k}}(v_i,x)$ for $2\leq i\leq n-k$.        

\item[(ii)] Let $x,y$ be  pendant vertices at $v_1$. Then $f_{K^k_{n-k}}(x,y)=2^{k-2}f_{K_{n-k}}(v_1)$. For $j\geq 2$, we have $f_{K^k_{n-k}}(x,v_j)=f_{K^k_{n-k}}(x,v_1,v_j)<f_{K^k_{n-k}}(x,v_1)$ and $f_{K^k_{n-k}}(x,v_j)=2^{k-1}f_{K_{n-k}}(v_1,v_j)$. Thus  $\epsilon(x)=\min\{f_{K^k_{n-k}}(x,y),f_{K^k_{n-k}}(x,v_j)\}$ where $j\neq 1$. 
\end{enumerate}  
\end{proof}
\end{lemma}
\begin{theorem}\label{thm6}
Let $2\leq k\leq n-3$ and  let $G\in\mathcal{G}^k_n$. Then $ES(G)\leq ES(K^k_{n-k})$ and equality happens if and only if $G\cong K^k_{n-k}$.
\end{theorem}
\begin{proof}
Suppose $G$ is not isomorphic to $K^k_{n-k}$. We have two cases:\\
{\bf Case-I:} $G\in G^k_n(S)$\\
We can obtain $K^k_{n-k}$ from $G$ by deleting all the pendant vertices of $G$ and attaching them as pendant vertices to $v_1$.
Since $G$ is not isomorphic to $K^k_{n-k}$, for some $i\geq 2$ there exists a vertex $v_i$ in $G$ with degree at least $n-k$. Let $z$ be a pendant vertex adjacent to $v_i$ in $G$. Let $x$ be a pendant vertex adjacent to $v_1$ in $G$. Then $\epsilon_G(z)\leq f_G(z,x)\leq 2^{k-2}f_{K_{n-k}}(v_i,v_1)$. By Lemma \ref{lem14}(ii), we have $\epsilon_{K^k_{n-k}}(z)=\min\{f_{K^k_{n-k}}(z,x),f_{K^k_{n-k}}(z,v_i)\}=\min\{2^{k-2}f_{K_{n-k}}(v_1),2^{k-1}f_{K_{n-k}}(v_1,v_i)\}>2^{k-2}f_{K_{n-k}}(v_i,v_1)\geq \epsilon_G(z)$. Thus $\epsilon_G(z)<\epsilon_{K^k_{n-k}}(z)$. Also 
\begin{align*}
\epsilon_G(v_1)&\leq f_G(v_1,z)\leq 2^{k-1}f_{K_{n-k}}(v_1,v_i)\\
&< \min\{2^kf_{K_{n-k}}(v_1,v_i),2^{k-1}f_{K_{k-n}}(v_1)\}\\
&=\min\{f_{K^k_{n-k}}(v_1,v_i), f_{K^k_{n-k}}(v_1,x)\}\\
&=\epsilon_{K^k_{n-k}}(v_1) \;\; \mbox{(by Lemma \ref{lem14}(i)).}
\end{align*}
Similarly, it can be shown that $\epsilon_G(v)\leq \epsilon_{K^k_{n-k}}(v)$ for any $v\in V(G)\setminus \{v_1,z\}$ using Lemma \ref{lem14}. Thus $ES(G)< ES(K^k_{n-k})$.

\noindent {\bf Case-II:} $G\notin G^k_n(S)$\\
By Lemma \ref{lem13}, there exists a $G'\in G^k_n(S)$ such that $ES(G)<ES(G')$ and then by Case-I, we have  $ES(G)<ES(G')\leq ES(K^k_{n-k})$. This proves the result.
\end{proof}
  
\section{Conclusions}

In this work, we started the study of obtaining the extremal graphs associated with eccentric subgraph index. We obtain the graphs which maximizes the eccentric subgraph index over all unicyclic graphs, over all unicyclic graphs with fixed girth and over all graphs with fixed number of pendant vertices. We also characterized the trees which maximizes the eccentric subgraph index over all trees on $n$ vertices with diameter $d$. It will be interesting to know the graphs which minimizes eccentric subgraph index over the above classes of graphs or trees.\\

\noindent{\bf Conflict of interest:} On behalf of all authors, the corresponding author states that there is no conflict of interest.\\

\noindent{\bf Funding:} The author Akash De is supported by the Council of Scientific and Industrial Research(CSIR) Grant No. 09/1002(17117)/2023-EMR-I.\\

\noindent{\bf Data availability:} No data was used for the research described in the article.\\

\noindent{\bf Declaration of generative AI use:} The authors report generative AI was not used in their research or preparation of this manuscript.

\end{document}